\documentclass[a4paper]{amsart}
\usepackage[ps2pdf]{hyperref}
\usepackage{amsmath,amssymb}
\usepackage{bbold}
\usepackage{psfrag}
\usepackage{graphicx}
\usepackage{leftidx}
\usepackage[all,poly]{xy}
\usepackage[latin1]{inputenc}
\usepackage{tikz-cd}

\newtheorem{thm}{Theorem}[section]
\newtheorem{cor}[thm]{Corollary}
\newtheorem{lem}[thm]{Lemma}
\theoremstyle{definition}
\newtheorem{exa}[thm]{Example}
\newtheorem{rk}[thm]{Remark}

\newcommand{\chr}{{\mathsf{c}}}
\newcommand{\Proj}{\mathrm{Proj}}
\newcommand{\mt}{\mathsf{t}}
\newcommand{\co}{\colon}
\newcommand{\id}{\mathrm{id}}
\newcommand{\opp}{\mathrm{op}}
\newcommand{\cc}{\mathcal{C}}
\newcommand{\bb}{\mathcal{B}}
\newcommand{\aaa}{\mathcal{A}}
\newcommand{\zz}{\mathcal{Z}}
\newcommand{\kk}{\Bbbk}
\newcommand{\N}{\mathbb{N}}
\newcommand{\un}{\mathbb{1}}
\newcommand{\End}{\mathrm{End}}
\newcommand{\Hom}{\mathrm{Hom}}
\newcommand{\lev}{\mathrm{ev}}
\newcommand{\rev}{\widetilde{\mathrm{ev}}}
\newcommand{\lcoev}{\mathrm{coev}}
\newcommand{\rcoev}{\widetilde{\mathrm{coev}}}
\newcommand{\ldual}[1]{\leftidx{^\vee}{\!#1}{}}
\newcommand{\rdual}[1]{{#1}^\vee}
\newcommand{\lldual}[1]{\leftidx{^{\vee\vee}}{\!#1}{}}
\newcommand{\rrdual}[1]{{#1}^{\vee\vee}}
\newcommand{\bdual}[1]{{#1}^{**}}
\newcommand{\rsdraw}[3]{\raisebox{-#1\height}{\scalebox{#2}{\includegraphics{#3.eps}}}}

\newcommand{\labeli}{\renewcommand{\labelenumi}{{\rm (\roman{enumi})}}}
\providecommand{\bysame}{\leavevmode\hbox to3em{\hrulefill}\thinspace}

\begin{document}

\title{Chromatic maps for finite tensor categories}

\author[F. Costantino]{Francesco Costantino}
\address{Institut de Math\'ematiques de Toulouse\\
118 route de Narbonne\\
 Toulouse F-31062}
\email{francesco.costantino@math.univ-toulouse.fr}

\author[N. Geer]{Nathan Geer}
\address{Mathematics \& Statistics\\
  Utah State University \\
  Logan, Utah 84322, USA}
\email{nathan.geer@gmail.com}

\author[B. Patureau-Mirand]{Bertrand Patureau-Mirand} \address{Univ
  Bretagne Sud, CNRS UMR 6205, LMBA, F-56000 Vannes, France}
\email{bertrand.patureau@univ-ubs.fr}

\author[A. Virelizier]{Alexis Virelizier} \address{Univ. Lille, CNRS, UMR 8524 - Laboratoire Paul Painlev\'e, F-59000 Lille, France}
\email{alexis.virelizier@univ-lille.fr}

\date{\today}

\begin{abstract}
Chromatic maps for spherical tensor categories are instrumental tools to construct (non semisimple) invariants of 3-manifolds and their extension to (non compact) (2+1)-TQFTs. In this paper, we introduce left and right chromatic maps for finite tensor categories and prove that such maps always exist. As a corollary, we obtain that any spherical finite tensor category has a chromatic map.
\end{abstract}

\maketitle

\setcounter{tocdepth}{1}
\tableofcontents

\section{Introduction}\label{sect-intro}
During the last 30 years, deep relations have emerged between low-dimensional topology and the theory of monoidal categories.
In particular, monoidal categories carrying appropriate additional structures (such as pivotality, linearity, finiteness, braidings, etc...) have been used to construct topological invariants of 3-manifolds and (2+1)-TQFTs, see for example~\cite{Tu0,TVi5} for the semisimple approach, \cite{KL} for a first non semisimple approach, and~\cite{CGPT20,CGPV23} for a new non semisimple approach. In this latter non semisimple approach,  the main algebraic tools used for the topological constructions are the chromatic maps whose definition involves the (non degenerate) modified trace on the ideal of projective objects of a spherical finite tensor category. The chromatic maps are also a key ingredient for the definition of the skein (3+1)-TQFTs of \cite{CGHP}. (In both cases, the chromatic
maps play the role of so called ``Kirby colors'' in the surgery semisimple approach  see Example~\ref{ex-spherical-fusion} or of the integral for Hennings invariants, see Example~\ref{ex-Hopf-chromatic}.)

In this paper, given a finite tensor category $\cc$ (in the sense of~\cite{EGNO}), we introduce right and left chromatic maps for $\cc$. Their definition involves the distinguished invertible object of $\cc$ (see Section~\ref{subsect-chromatic}). Our main result (Theorem~\ref{thm-chromatic}) is that left and right chromatic maps for $\cc$ always exist. The proof of this result uses the central Hopf monad (which describes the center of $\cc$, see~\cite{BV3}) and the existence and uniqueness of (co)integrals based at the distinguished invertible object of $\cc$ (see~\cite{Sh2}).

As a corollary, we get  (Corollary~\ref{thm-chromatic}) that any spherical finite tensor category has a chromatic map in the sense of~\cite{CGPT20,CGPV23}. In particular, this proves that the construction of invariants of closed 3-manifolds given in~\cite{CGPT20} and their extension to non-compact (2+1)-TQFTs given in~\cite{CGPV23} can be performed starting with any spherical finite tensor category.

The paper is organized as follows. In Section~\ref{sect-chromatic}, after some algebraic preliminaries, we introduce right and left chromatic maps for a finite tensor category $\cc$ and state their existence. Next, when $\cc$ is spherical, we relate them to the chromatic maps in the sense of~\cite{CGPT20,CGPV23}. Section~\ref{sect-proofs} is dedicated to the proofs of the above results. In particular, we recall the notion of a Hopf monad and its based (co)integrals as well as the construction (in terms of coends) of the central Hopf monad which are used in the proof of the main result.

\section{Chromatic maps}\label{sect-chromatic}

Throughout the paper, we fix an algebraically closed field~$\kk$.

\subsection{Conventions on monoidal categories}\label{sect-conventions}
For the basics on monoidal categories, we refer for example to \cite{ML1,EGNO,TVi5}. We  will suppress in our formulas  the associativity  and unitality  constraints   of monoidal categories.   This does not lead  to   ambiguity because by   Mac Lane's   coherence theorem,   all legitimate ways of inserting these constraints   give the same result.
For  any objects $X_1,...,X_n$ with $n\geq 2$, we set
$$
X_1 \otimes X_2 \otimes \cdots \otimes X_n=(... ((X_1\otimes X_2) \otimes X_3) \otimes
\cdots \otimes X_{n-1}) \otimes X_n
$$
and similarly for morphisms.

Recall that a monoidal category is \emph{rigid} if every object admits a left dual and a right dual. Subsequently, when dealing with rigid categories, we shall always assume tacitly that for each object $X$, a left dual $\ldual{X}$ and a right dual $\rdual{X}$  has been chosen, together with their (co)evaluation morphisms
\begin{align*}
& \lev_X\co \ldual{X} \otimes X \to \un , && \lcoev_X \co \un \to X \otimes \ldual{X},\\
& \rev_X\co X \otimes \rdual{X} \to \un, && \rcoev_X \co \un \to \rdual{X} \otimes X.
\end{align*}

A \emph{pivotal category} is a rigid monoidal category with a choice of left and right duals for objects so that the induced  left and right dual functors coincide as monoidal functors. Then we write $X^*=\ldual{X}=\rdual{X}$ for any $X \in \cc$, and
$$\phi=\{\phi_X=(\id_{X^{**}} \otimes \lev_X)(\rcoev_{X^*} \otimes \id_X)\co X \to X^{**}\}_{X \in \cc}$$ is a
monoidal natural isomorphism relating the (co)evaluation morphisms, called the \emph{pivotal structure} of $\cc$.

\subsection{Projective objects and covers}\label{sect-abelian}
An object $P$ of an abelian category $\cc$ is \emph{projective} if the functor $\Hom(P,-)\colon \cc \to \mathrm{Set}$
preserves epimorphisms.  An abelian category has \emph{enough projectives} if every object has an epimorphism from a projective object onto it.

A \emph{projective cover} of an object $X$ of an abelian category $\cc$ is a projective object $P(X)$ of $\cc$ together with an epimorphism $p \co P(X) \to X$ such that if $g \co P \to X$ is an epimorphism from a projective object $P$ to $X$, then there exists an epimorphism $h \co P \to P(X)$ such that $ph = g$. Note that if it exists, a projective cover is unique up to a non-unique isomorphism, and a projective cover of a simple object is indecomposable.

\subsection{Finite tensor categories}\label{sect-finite-tensor}
Recall that a monoidal category is \emph{$\kk$-linear} if each hom-set carries a structure of a $\kk$-vector space so that  the composition and monoidal product of morphisms are $\kk$-bilinear.

Following \cite{EGNO}, a \emph{finite tensor category} (over $\kk$) is a $\kk$-linear rigid monoidal abelian category $\cc=(\cc,\otimes,\un)$ such that:
\begin{enumerate}
\labeli
\item the category $\cc$ is locally finite: the hom-sets are finite dimensional and every object of~$\cc$ has finite length,
\item the category $\cc$ has enough projectives,
\item there are finitely many isomorphism classes of simple objects,
\item $\End_\cc(\un)=\kk\, \id_\un$.
\end{enumerate}

Let $\cc$ be a finite tensor category. Then the unit object $\un$ of $\cc$ is simple (see \cite[Theorem 4.3.8]{EGNO}). Also, every simple object of $\cc$ has a projective cover, and any indecomposable projective object $P$ of $\cc$ has a unique simple subobject, called the \emph{socle} of $P$ (see \cite[Remark 6.1.5]{EGNO}). In particular, the socle of the projective cover of the unit object $\un$ is an invertible object called the \emph{distinguished invertible} object of $\cc$.
Finally note that $\cc$ has a \emph{projective generator}, that is, a projective object $G$ such that the functor $\Hom_\cc(G,-)\co \cc \to \mathrm{Set}$ is faithful or, equivalently, such that any projective object is a retract of $G^{\oplus n}$ for some $n\in\N$. 

\subsection{Chromatic maps}\label{subsect-chromatic}
Let $\cc$ be a finite tensor category. Pick a projective cover $\varepsilon\co P_0 \to \un$  of the unit object and a  monomorphism $\eta\co \alpha \to P_0$, where $\alpha$ is the distinguished invertible object of $\cc$.

\begin{lem}\label{lem-Lambda}
There are unique natural transformations
$$\Lambda^r=\{\Lambda_X^r \co \alpha \otimes X \to X\}_{X \in \cc}
\quad \text{and} \quad
\Lambda^l=\{\Lambda_X^l \co X \otimes \alpha \to X\}_{X \in \cc}
$$
such that for any indecomposable projective object $P$ non isomorphic to $P_0$,
$$
\Lambda^r_P=0, \quad  \Lambda^l_P=0,  \quad \text{and} \quad \Lambda^r_{P_0}=\eta \otimes \varepsilon, \quad \Lambda^l_{P_0}=\varepsilon \otimes \eta.
$$
\end{lem}
We prove Lemma~\ref{lem-Lambda} in Section~\ref{proof-lem-Lambda}.

Let $P$ be a projective object and $G$ be a projective generator of~$\cc$.  A \emph{right chromatic map} based at $P$ for $G$ is a morphism
$$
\chr_P^r\in\Hom_\cc(P \otimes \rrdual{G}, P \otimes G \otimes \alpha )
$$
such that for all $X\in\cc$,
$$
(\id_P \otimes \rev_G \otimes  \id_X)(\id_{P \otimes G} \otimes \Lambda^{r}_{\rdual{G} \otimes X})(\chr_P^r \otimes \id_{\rdual{G} \otimes X}) (\id_P \otimes \rev_{\rdual{G}} \otimes \id_X)=\id_{P \otimes X}.
$$
Using graphical calculus for monoidal categories (with the convention of diagrams to be read from bottom to top), the latter condition depicts as:
$$
\psfrag{P}[Br][Br]{\scalebox{.8}{$P$}}
\psfrag{G}[Br][Br]{\scalebox{.8}{$G$}}
\psfrag{a}[Br][Br]{\scalebox{.8}{$\alpha$}}
\psfrag{H}[Bl][Bl]{\scalebox{.8}{$\rrdual{G}$}}
\psfrag{X}[Bl][Bl]{\scalebox{.8}{$X$}}
\psfrag{Q}[Bl][Bl]{\scalebox{.8}{$P$}}
\psfrag{R}[Br][Br]{\scalebox{.8}{$\rdual{G}$}}
\psfrag{u}[Bc][Bc]{\scalebox{.9}{$\rev_{G}$}}
\psfrag{e}[Bc][Bc]{\scalebox{.9}{$\rcoev_{\rdual{G}}$}}
\psfrag{c}[Bc][Bc]{\scalebox{.9}{$\chr_P^r$}}
\psfrag{A}[Bc][Bc]{\scalebox{.9}{$\Lambda^r_{\rdual{G} \otimes X}$}}
\rsdraw{.45}{.9}{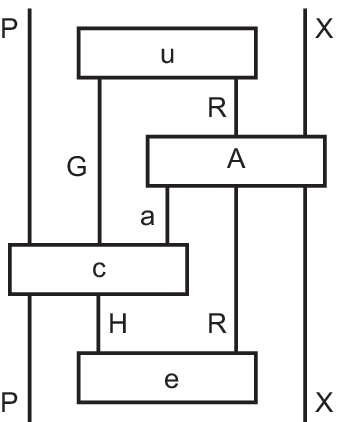}\; =\;\;\, \rsdraw{.45}{.9}{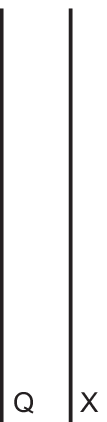}\;.
$$
Similarly, a \emph{left chromatic map} based at $P$ for $G$ is a morphism
$$
\chr_P^l\in\Hom_\cc(\lldual{G} \otimes P, \alpha \otimes G \otimes P)
$$
such that for all $X\in\cc$,
$$
(\id_X \otimes \lev_G \otimes \id_P)(\Lambda^l_{X \otimes \ldual{G}} \otimes \id_{G \otimes P})(\id_{X \otimes \ldual{G}} \otimes \chr_P^l) (\id_X \otimes \lcoev_{\ldual{G}} \otimes \id_P)=\id_{X \otimes P}.
$$
This condition depicts as:
$$
\psfrag{P}[Bl][Bl]{\scalebox{.8}{$P$}}
\psfrag{G}[Bl][Bl]{\scalebox{.8}{$G$}}
\psfrag{a}[Bl][Bl]{\scalebox{.8}{$\alpha$}}
\psfrag{H}[Br][Br]{\scalebox{.8}{$\lldual{G}$}}
\psfrag{X}[Br][Br]{\scalebox{.8}{$X$}}
\psfrag{Y}[Bl][Bl]{\scalebox{.8}{$X$}}
\psfrag{Q}[Bl][Bl]{\scalebox{.8}{$P$}}
\psfrag{R}[Bl][Bl]{\scalebox{.8}{$\ldual{G}$}}
\psfrag{u}[Bc][Bc]{\scalebox{.9}{$\lev_{G}$}}
\psfrag{e}[Bc][Bc]{\scalebox{.9}{$\lcoev_{\ldual{G}}$}}
\psfrag{c}[Bc][Bc]{\scalebox{.9}{$\chr_P^l$}}
\psfrag{A}[Bc][Bc]{\scalebox{.9}{$\Lambda^l_{X \otimes \ldual{G}}$}}
\rsdraw{.45}{.9}{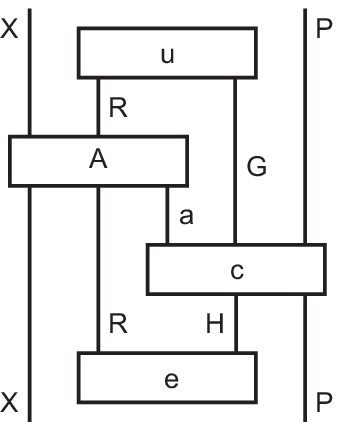}\; =\;\;\, \rsdraw{.45}{.9}{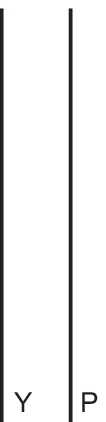}\;.
$$

The main result of the paper is the existence of right and left chromatic maps for finite tensor categories:
\begin{thm}\label{thm-chromatic}
For any projective object $P$ and any projective generator $G$ of~$\cc$, there are a right chromatic map and a left chromatic map based at $P$ for $G$.
\end{thm}
We prove Theorem~\ref{thm-chromatic} in Section~\ref{sect-proof-chromatic}.

\begin{exa}\label{exa-H-mod}
Let $H$ be a finite dimensional Hopf algebra over $\kk$. The category $H$\nobreakdash-mod of finite dimensional (left) $H$-modules and $H$-linear homomorphisms is a finite tensor category. Recall that the left dual of an object $M$ of $H$-mod  is the $H$-module $\ldual{M}=M^\ast=\Hom_\kk(M,\kk)$  where each $h \in H$ acts as the transpose of $m \in M \mapsto S(h)\cdot m \in M$, with~$S$  the antipode of $H$. The right dual of $M$ is $\rdual{M}=M^\ast$  where each $h \in H$ acts as the transpose of $m \in M \mapsto S^{-1}(h)\cdot m \in M$. The associated left and right evaluation morphisms are computed for any $m \in M$ and $\varphi \in M^\ast$ by
$$
\lev(\varphi \otimes m)=\varphi(m)=\rev(m \otimes \varphi).
$$
A projective generator of $H$-mod is $H$ equipped with its left
regular action. It follows from \cite[Proposition 6.5.5.]{EGNO} that
the distinguish object $\alpha$ of $H$-mod is~$\kk$ with action
$H \otimes \hspace{.1em} \kk \cong H \to \kk$ given by the inverse
$\alpha_{H}\in H^*$ of the distinguished grouplike element of $H^*$.
(The form $\alpha_H$ is characterized by
$\Lambda S(h)=\alpha_H(h) \Lambda$ for all $h \in H$ and all left
cointegral $\Lambda\in H$.)  Pick a projective cover $\varepsilon\co P_0 \to \kk$  of the unit object and a  monomorphism $\eta\co \alpha \to P_0$.
Since the counit
$\varepsilon_H\co H\to\kk$ of $H$ is an epimorphism, there exists an epimorphism
$p\co H\to P_0$ such that $\varepsilon_H=\varepsilon p$. Let $i\co P_0\to H$
be a section of $p$ in $H$-mod and set $\Lambda=S(i\eta(1_\kk))\in H$.
We prove in Section~\ref{proof-chromatic-H-mod} that $\Lambda$ is
a non zero left cointegral. By \cite[Theorem 10.2.2]{R2012}, there is a unique right integral
$\lambda\in H^*$ such that $\lambda(\Lambda)=1$.  
Then a left chromatic map based at $H$ for $H$ is
$$
 \chr_H^l \co \left \{\begin{array}{ccl} \lldual{H}\otimes H & \to & \alpha\otimes H\otimes H \\ e_x \otimes y & \mapsto &  \lambda\bigl(S(y_{(1)})x\bigr)\alpha_{H}(y_{(2)})\otimes y_{(3)}\otimes y_{(4)}
  \end{array}\right. 
$$
and a right chromatic map based at $H$ for $H$ is 
$$
 \chr_H^r \co \left \{\begin{array}{ccl} H\otimes \rrdual{H} & \to & H\otimes H\otimes \alpha \\ y \otimes e_x & \mapsto & 
  y_{(1)}\otimes y_{(2)} \otimes \alpha_{H}(y_{(2)})  \lambda\bigl(S(x)y_{(1)}\bigr)
  \end{array}\right. 
$$
where $x \in H \mapsto e_x \in H^{\ast\ast}$ is the canonical $\kk$-linear isomomorphism. We verify this in Section~\ref{proof-chromatic-H-mod}.
More generally, for any finite dimensional projective $H$-module~$P$, 
$$
 \chr_P^l= \sum_i (\id_{\alpha \otimes H} \otimes g_i) \chr_H^l (\id_{\lldual{H}} \otimes f_i) \quad \text{and} \quad \chr_P^r=\sum_i (g_i \otimes \id_{H \otimes \alpha}) \chr_H^l (\id_{f_i \otimes \rrdual{H}}) 
$$
are respectively a left chromatic map and right chromatic map based on $P$ for $H$, where $\{f_i\co P\to H,g_i \co H\to P\}_i$ is any finite family of $H$-linear homomorphisms such that $\id_P=\sum_ig_if_i$.
\end{exa}

\subsection{The case of spherical finite tensor categories}\label{sect-chromatic-pivotal}
Recall the notion of a left or right modified trace from \cite{GPV13}.
Following \cite{SS2021}, a finite tensor category is \emph{spherical} if it is pivotal, unimodular (meaning that $\un$ is the distinguished invertible object), and the unique non degenerate right modified trace on the full subcategory $\Proj$ of projective objects is also a left modified trace.

Let~$\cc$ be a spherical finite tensor category. Pick a  projective cover $\varepsilon \co P_0 \to \un$ of the monoidal unit and a mono\-mor\-phism $\eta \co \un \to P_0$.
Consider the unique non degenerate two-sided modified trace $$\mt=\{\mt_P \co \End_\cc(P) \to \kk\}_{P \in \Proj}$$ such that $\mt_{P_0}(\eta \varepsilon)=1_\kk$ (see~\cite[Corollary 5.6]{GKP22}). Recall that the non degeneracy of~$\mt$ implies that the pairing $u \otimes_\kk v \in \Hom_\cc(P,\un)\otimes_\kk \Hom_\cc(\un,P)\to \mt_P(vu)\in\kk$ is non degenerate.
Define
 $$\Lambda^\mt=\{\Lambda_P^\mt \co P \to P\}_{P \in \Proj}$$
by setting $\Lambda_P^\mt =\sum_i x_i \circ x^i$, where $(x^i)_i$ and $(x_i)_i$ are basis of $\Hom_\cc(P,\un)$ and $\Hom_\cc(\un,P)$ which are dual with respect to above pairing. Note that the cyclicity of $\mt$ (meaning that $\mt_P(gf)=\mt_Q(fg)$ for all morphisms $f \co P \to Q$ and $g \co Q \to P$ in $\Proj$) implies that the family $\Lambda^\mt$ is natural in $P \in \Proj$.

Let $P$ be a projective object and  $G$ be a projective generator of $\cc$. A \emph{chromatic map} based on $P$ for $G$ in the sense of~\cite{CGPT20,CGPV23} is a morphism $$\chr_P\in\End_\cc(G\otimes P)$$ such that for all $X\in\cc$,
$$
(\id_X \otimes \lev_G \otimes \id_P)(\Lambda^{\mt}_{X \otimes G^*} \otimes \chr_p) (\id_X \otimes \rcoev_G \otimes \id_P)=\id_{X \otimes P}.
$$
This condition depicts as:
$$
\psfrag{P}[Bl][Bl]{\scalebox{.8}{$P$}}
\psfrag{G}[Bl][Bl]{\scalebox{.8}{$G$}}
\psfrag{X}[Br][Br]{\scalebox{.8}{$X$}}
\psfrag{Y}[Bl][Bl]{\scalebox{.8}{$X$}}
\psfrag{c}[Bc][Bc]{\scalebox{.9}{$\chr_P$}}
\psfrag{A}[Bc][Bc]{\scalebox{.9}{$\Lambda^\mt_{X \otimes G^*}$}}
\rsdraw{.45}{.9}{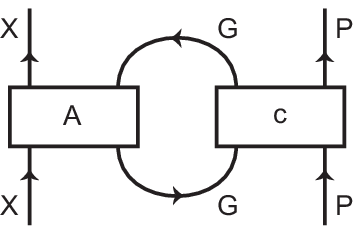}\; =\;\; \rsdraw{.45}{.9}{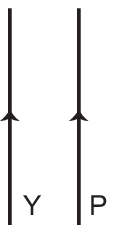}\;.
$$
Here we use the following graphical conventions for pivotal categories: strands are oriented, a strand oriented upwards (respectively, downwards) and colored by an object $V \in \cc$ corresponds to $\id_V$ (respectively, $\id_{V^*}$), and the oriented caps and cups correspond to the evaluation and coevaluation morphisms.

Consider the natural transformations $\Lambda^r$ and $\Lambda^l$ associated with $\varepsilon$ and $\eta$ as in~Lemma~\ref{lem-Lambda}.
\begin{lem}\label{lem-Lambdas-comparison}
For any projective object $P$ of $\cc$, we have: $\Lambda^\mt_P=\Lambda^r_P=\Lambda^l_P$.
\end{lem}
We prove Lemma~\ref{lem-Lambdas-comparison} in Section~\ref{proof-lem-Lambdas-comparison}. The next corollary is a direct consequence of Theorem~\ref{thm-chromatic} and Lemma~\ref{lem-Lambdas-comparison}.

\begin{cor}\label{cor-chromatic}
Any spherical tensor category has chromatic maps.
\end{cor}
\begin{proof}
Lemma~\ref{lem-Lambdas-comparison} and the fact that $\rcoev_G=(\id_{G^*} \otimes \phi^{-1}_{G})\lcoev_{G^*}$,
where $\phi$ is the pivotal structure of $\cc$ (see Section~\ref{sect-conventions}) imply that a morphism $\chr_P\co G \otimes P \to G \otimes P$ is a chromatic map  if and only if the morphism $\chr_P^l= \chr_P( \phi_G^{-1} \otimes \id_P) \co \bdual{G} \otimes P \to G \otimes P$
is a left chromatic map.
The existence of a chromatic map based on any projective object $P$ for any projective generator $G$ follows then from Theorem~\ref{thm-chromatic}.
\end{proof}

\begin{rk}
Let $P$ be a projective object and  $G$ be a projective generator of~$\cc$. Then $P^*$ is a projective object and  $G^*$ is a projective generator of $\cc$. By Corollary~\ref{cor-chromatic}, there is a chromatic map $\chr_{P^*}\in \End_\cc(G^* \otimes P^*)$ based on $P^*$ for $G^*$. Then the morphism
$$
\psfrag{P}[Br][Br]{\scalebox{.8}{$P$}}
\psfrag{G}[Bl][Bl]{\scalebox{.8}{$G$}}
\psfrag{c}[Bc][Bc]{\scalebox{.9}{$\chr_{P^*}$}}
\tilde{\chr}_{P}=\rsdraw{.45}{.9}{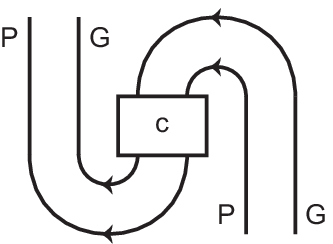}\in\End_\cc(P\otimes G)
$$
satisfies the following relation: for all $X\in\cc$,
$$
\psfrag{P}[Br][Br]{\scalebox{.8}{$P$}}
\psfrag{G}[Br][Br]{\scalebox{.8}{$G$}}
\psfrag{X}[Bl][Bl]{\scalebox{.8}{$X$}}
\psfrag{R}[Bl][Bl]{\scalebox{.8}{$P$}}
\psfrag{c}[Bc][Bc]{\scalebox{.9}{$\tilde{\chr}_{P}$}}
\psfrag{A}[Bc][Bc]{\scalebox{.9}{$\Lambda^\mt_{G^* \otimes X}$}}
\rsdraw{.45}{.9}{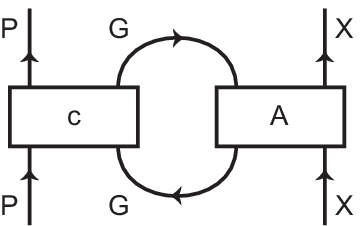}\; =\;\; \rsdraw{.45}{.9}{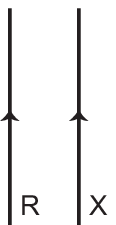}\;.
$$
\end{rk}

\begin{exa}\label{ex-spherical-fusion}
Let $\cc$ be a spherical fusion  category.  Then every object of $\cc$  is projective and a projective generator is $G=\bigoplus_{i\in I} i$, where $I$ is a representative set of simples object of $\cc$. For any projective object $P$, a chromatic map based on $P$ is 
$$
\chr_P=\bigoplus_{i\in I} \mathrm{qdim}(i) \, \id_i\otimes \id_P.
$$
Formally, $\chr_P=\id_\Omega \otimes \id_P$, where $\Omega=\bigoplus_{i\in I} \mathrm{qdim}(i) \,i$ is the ``Kirby color'' of $\cc$. 
\end{exa}

\begin{exa}\label{ex-Hopf-chromatic}
Let $H$ be a finite dimensional Hopf algebra over $\kk$. Then the category $H$-mod (see Example~\ref{exa-H-mod}) is spherical if and only if $H$ is unimodular and unibalanced in the sense of \cite{BBG} (meaning that $H$ is pivotal with pivot a square root of the distinguish grouplike element of $H$). Assume this is the case. With the notation of Example~\ref{exa-H-mod}, this means
that $\alpha_{H}\in H^*$ is the counit of $H$ and the pivot $g$ of $H$ satisfies $S^2(h)=ghg^{-1}$ and $\lambda(h_{(2)})h_{(1)}=\lambda(h)g^2$ for all $h\in H$.  Then the canonical $\kk$-linear isomomorphism $x \in H \mapsto e_x \in H^{\ast\ast}$ induces a $H$-linear isomorphism $H\cong \lldual{H}$ equal to the pivotal structure of $H$-mod evaluated at $H$. Consequently, it follows from Example~\ref{exa-H-mod} (see also \cite[Section 6]{CGPT20}) that a chromatic map based at $H$ for $H$~is
$$
 \chr_H \co \left \{\begin{array}{ccl} H\otimes H & \to & H\otimes H \\ x \otimes y & \mapsto & \lambda(S(y_{(1)})gx)\, y_{(2)}\otimes y_{(3)}.
  \end{array}\right. 
$$
More generally, for any finite dimensional projective $H$-module~$P$, 
$$
 \chr_P= \sum_i (\id_H \otimes g_i) \chr_H (\id_H \otimes f_i) \co H \otimes P \to H \otimes P
$$
is a chromatic map based on $P$ for $H$, where $\{f_i\co P\to H,g_i \co H\to P\}_i$ is any finite family of $H$-linear homomorphisms such that $\id_P=\sum_ig_if_i$.
\end{exa}

\section{Proofs}\label{sect-proofs}

We first prove Lemma~\ref{lem-Lambda} in Section~\ref{proof-lem-Lambda}. Next we recall the notions of a Hopf monad and their based (co)integrals (Section~\ref{sect-Hopf-monads}) and the construction of the central Hopf monad (Section~\ref{sect-central-Hopf-monad}). Then we use these notions to prove Theorem~\ref{thm-chromatic} in Section~\ref{sect-proof-chromatic}.
Finally, in Section~\ref{proof-lem-Lambdas-comparison}, we prove Lemma~\ref{lem-Lambdas-comparison}.

\subsection{Proof of Lemma~\ref*{lem-Lambda}}\label{proof-lem-Lambda}
For any indecomposable projective object $P$ non isomorphic to $P_0$ and all morphisms $f \in \End_\cc(P_0)$, $g \in \Hom_\cc(P,P_0)$, $h \in \Hom_\cc(P_0,P)$, it follows from \cite[Lemma 4.3]{GKP22} that
$$
\eta \otimes \varepsilon f=f\eta \otimes \varepsilon, \quad \varepsilon g=0, \quad h\eta=0.
$$
This and the fact that any projective object is a (finite) direct sum of indecomposable projective objects imply that the prescriptions of Lemma~\ref{lem-Lambda} uniquely define natural transformations
$\{\Lambda_P^r \co \alpha \otimes P \to P\}_{P \in \Proj}$
and $\{\Lambda_P^r \co P \otimes \alpha \to P\}_{P \in \Proj}$, where $\Proj$ is the full subcategory of $\cc$ of projective objects. These natural transformations further uniquely  extend to $\cc$ by applying the next Lemma~\ref{lem-extension} with the functor $F=\alpha \otimes -$ (which is exact since it is an equivalence because $\alpha$ is invertible)
and the identity functor $G=1_\cc$.

\begin{lem}\label{lem-extension}
Let $F,G \co \aaa \to \bb$ be additive functors between abelian categories. Assume that $\aaa$ has enough projectives and that $F$ is right exact. Denote by  $\Proj$ the full subcategory of $\aaa$ of projective objects.
Then any natural transformation
$\{\alpha_P \co F(P) \to G(P)\}_{P \in \Proj}$ uniquely extends to $\aaa$, that is, to a natural transformation $\{\alpha_X \co F(X) \to G(X)\}_{X \in \aaa}$.
\end{lem}
\begin{proof}
Consider a natural transformation $\alpha=\{\alpha_P \co F(P) \to G(P)\}_{P \in \Proj}$. Assume first that $\bar{\alpha}$ and $\tilde{\alpha}$ are both extensions of $\alpha$ to $\aaa$. Let $X \in \aaa$. Pick an
epimorphism $p\co P \to X$ with $P$ projective.
Using the  naturality of $\bar{\alpha}$ and $\tilde{\alpha}$ together with the fact that both $\bar{\alpha}$ and $\tilde{\alpha}$ extend $\alpha$, we have:
$$
\bar{\alpha}_X F(p)=G(p) \bar{\alpha}_P=G(p) \alpha_P=G(p) \tilde{\alpha}_P=\tilde{\alpha}_X F(p).
$$
Thus $\bar{\alpha}_X=\tilde{\alpha}_X$ since $F(p)$ is an epimorphism (because $p$ is and $F$ is right exact). This proves the uniqueness of an extension of $\alpha$ to $\aaa$.

We now prove the existence of an extension of $\alpha$ to $\aaa$. Let $X \in \aaa$. Pick an epimorphism $p\co P \to X$ with $P$ projective. Then there is a unique morphism $\bar{\alpha}_X \co F(X) \to G(X)$ in $\aaa$ such that
\begin{equation}\label{eq-p1}
\bar{\alpha}_X F(p)=G(p) \alpha_P.
\end{equation}
Indeed, since $\aaa$ is abelian, the epimorphism $p$ is the cokernel of its kernel $k \co K \to P$. Pick an epimorphism $r\co Q \to K$ with $Q$ projective. Then $p$ is the cokernel of $q=kr \co Q \to P$, and so $F(p)$ is the cokernel of $F(q)$ (because $F$ is right exact). Consequently, since
$$
G(p) \alpha_P F(q)=G(p)G(q) \alpha_Q= G(pq) \alpha_Q=G(0) \alpha_Q=0,
$$
there is a unique morphism $\bar{\alpha}_X \co F(X) \to G(X)$ in $\aaa$ satisfying \eqref{eq-p1}. Note that the morphism $\bar{\alpha}_X$ does not depend on the choice of $p$. Indeed, let $r \co R \to X$ be another epimorphism with $R$ projective and denote by $\tilde{\alpha}_X \co F(X) \to G(X)$ the unique morphism such that $G(r) \alpha_R=\tilde{\alpha}_X F(r)$. 
Since $P$ is projective and $r$ is an epimorphism, there is a morphism $s \co P \to R$ such that $p=rs$. Then
$$
\bar{\alpha}_X F(p)=G(p) \alpha_P=G(r)G(s) \alpha_P=G(r)\alpha_R F(s)=\tilde{\alpha}_X F(r)F(s)=\tilde{\alpha}_XF(p),
$$
and so $\tilde{\alpha}_X=\bar{\alpha}_X$  (since $F(p)$ is an epimorphism). Note also that $\bar{\alpha}_P=\alpha_P$ for all $P \in \Proj$.  Indeed, since $\id_P \co P \to P$ is an epimorphism with $P$ projective and using the defining relation \eqref{eq-p1}, we have:
$$
\bar{\alpha}_P= \bar{\alpha}_P F(\id_P)=G(\id_P) \alpha_P=\alpha_P.
$$
It remains to prove that the family $\bar{\alpha}=\{\bar{\alpha}_X \co F(X) \to G(X)\}_{X \in \aaa}$ is natural  in~$X$.
Let $f \co X \to Y$ be a morphism in $\aaa$. Pick epimorphisms $p\co P \to X$ and $q\co Q \to Y$ with $P,Q$ projective. Since $P$ is projective and $q$ is an epimorphism, there is a morphism $g \co P \to Q$ such that $fp=qg$.  Consider the following diagram:
\[\begin{tikzcd}[column sep={1.3em,between origins}, row sep={1.2em,between origins}]
	&&&&& {F(X)} &&&&&&& {F(Y)} \\ \\ \\ &&&&& {\small\text{(i)}} \\
	&&&&&&&&&&&& {\small\text{(ii)}} \\
	&&&&&& {F(Q)} \\ \\ {F(P)} &&&&&& {\small\text{(v)}} &&&&&& {G(Q)} &&&&& {G(Y)} \\ 	\\
	&&&&&& {G(P)} \\ &&&&&&&&&&&& {\small\text{(iii)}} \\
	&&&&& {\small\text{(iv)}} \\ \\ \\ 	&&&&& {F(X)} &&&&&&& {G(X)}
	\arrow["{F(p)}", from=8-1, to=1-6]
	\arrow["{F(f)}", from=1-6, to=1-13]
	\arrow["{\bar{\alpha}_Y}", from=1-13, to=8-18]
	\arrow["{F(p)}"', from=8-1, to=15-6]
	\arrow["{\bar{\alpha}_X}"', from=15-6, to=15-13]
	\arrow["{G(f)}"', from=15-13, to=8-18]
	\arrow["{F(q)}"', from=6-7, to=1-13]
	\arrow["{\alpha_P}"'{pos=0.7}, from=8-1, to=10-7]
	\arrow["{G(p)}", from=10-7, to=15-13]
	\arrow["{F(g)}"{pos=0.7}, from=8-1, to=6-7]
	\arrow["{G(q)}", from=8-13, to=8-18]
	\arrow["{\alpha_Q}"{pos=0.4}, from=6-7, to=8-13]
	\arrow["{G(g)}"'{pos=0.4}, from=10-7, to=8-13]
\end{tikzcd}\]
The inner squares (i) and (iii) commute by the functoriality of $F$ and $G$ applied to the equality $fp=qg$. The inner squares (ii) and (iv) commute by the defining relation \eqref{eq-p1}. The inner square (v) commutes by the naturality of $\alpha$. Consequently, the outer diagram commutes: $\bar{\alpha}_YF(f)F(p)=G(f)\bar{\alpha}_XF(p)$. Since $F(p)$ is an epimorphism (because $p$ is and $F$ is right exact), we obtain $\bar{\alpha}_YF(f)=G(f)\bar{\alpha}_X$.
\end{proof}

\subsection{Hopf monads and their based (co)integrals}\label{sect-Hopf-monads}
A \emph{monad} on a category $\cc$  is a monoid in the category of endofunctors of $\cc$, that
is, a triple $(T,m,u)$ consisting of a functor $T\co \cc \to \cc$ and two natural transformations
$$
m=\{m_X\co T^2(X) \to T(X)\}_{X \in \cc}\quad \text{and} \quad  u=\{u_X\co X \to T(X)\}_{X \in \cc}
$$
called the \emph{product} and the \emph{unit} of $T$, such that for any $X\in\cc$,
$$
m_XT(m_X)=m_Xm_{T(X)} \quad {\text {and}} \quad m_Xu_{T(X)}=\id_{T(X)}=m_X T(u_X).
$$

A \emph{bimonad} on monoidal category   $\cc$ is a monoid in the
category of  comonoidal endofunctors of $\cc$. In other words, a bimonad on $\cc$ is a
monad $(T,m,u)$ on $\cc$ such that the functor $T$ and the natural transformations $m$ and $u$ are comonoidal. The comonoidality of $T$ means that $T$ comes equipped with a natural transformation $ T_2=\{T_2(X,Y) \co  T(X \otimes Y)\to T(X)
\otimes T(Y)\}_{X,Y \in \cc} $ and a morphism $T_0\co T(\un) \to \un$ such that for all $X,Y,Z \in \cc$,
\begin{gather*}
\bigl(\id_{T(X)} \otimes T_2(Y,Z)\bigr) T_2(X,Y \otimes Z)= \bigl(T_2(X,Y) \otimes \id_{T(Z)}\bigr) T_2(X \otimes Y, Z) ,\\
(\id_{T(X)} \otimes T_0) T_2(X,\un)=\id_{T(X)}=(T_0 \otimes \id_{T(X)}) T_2(\un,X).
\end{gather*}
The comonoidality of $m$ and $u$ means that for all $X,Y \in~\cc$,
\begin{gather*}
T_2(X,Y)m_{X \otimes Y}=(m_X \otimes m_Y) T_2(T(X),T(Y))T(T_2(X,Y)),\\
 T_2(X,Y)u_{X \otimes Y}=u_X \otimes u_Y.
\end{gather*}

Let $T=(T,m,u)$ be a bimonad on a monoidal category $\cc$ and $A$ be an object of~$\cc$. A \emph{left $A$-integral} for~$T$ is a morphism $\Lambda_l \co T(A) \to \un$ in $\cc$ such that
$$
(\id_{T(\un)} \otimes \Lambda_l)T_2(\un,A)= u_\un  \Lambda_l.
$$
Similarly, a \emph{right $A$-integral} for $T$ is a morphism $\Lambda_r \co T(A) \to \un$ in $\cc$ such that
$$
(\Lambda_r \otimes \id_{T(\un)})T_2(A,\un)= u_\un  \Lambda_r.
$$
An \emph{$A$-cointegral} for $T$ is a morphism $\lambda \co \un \to T(A)$ in $\cc$ which is $T$-linear:
$$
m_A T(\lambda)=\lambda T_0.
$$

A \emph{Hopf monad} on  monoidal category $\cc$ is a bimonad on $\cc$
whose left and right fusion operators  are isomorphisms (see~\cite{BLV}).
When $\cc$ is a rigid category,    a bimonad $T$ on
$\cc$ is a Hopf monad if and only if it has a left antipode and a right antipode (see \cite{BV2}).  (Here, we will not need the actual definition of a Hopf monad and so just refer to~\cite{BLV,BV2}.)

\subsection{The central Hopf monad}\label{sect-central-Hopf-monad}
Let $\cc$ be a rigid monoidal category. Assume that for any $X \in \cc$, the coend
\begin{equation}\label{eq-coend-Z}
Z(X)=\int^{Y \in \cc} \ldual{Y} \otimes X \otimes Y
\end{equation}
exists. Denote by $i_{X,Y}\co \ldual{Y} \otimes X \otimes Y \to  Z(X)$ the associated universal dinatural transformation and set
\begin{equation*}
\partial_{X,Y}=(\id_{Y} \otimes i_{X,Y})(\lcoev_{Y} \otimes \id_{X \otimes Y}) \co X \otimes Y \to Y \otimes Z(X).
\end{equation*}
We will depict the morphism $\partial_{X,Y}$ as
$$
\psfrag{T}[Bl][Bl]{\scalebox{.8}{$Z(X)$}}
\psfrag{Z}[Br][Br]{\scalebox{.8}{$Y$}}
\psfrag{Y}[Bl][Bl]{\scalebox{.8}{$Y$}}
\psfrag{X}[Br][Br]{\scalebox{.8}{$X$}}
\partial_{X,Y}=\;\rsdraw{.45}{.9}{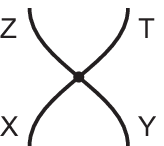}
$$
and call $\partial=\{\partial_{X,Y}\}_{X,Y \in \cc}$ the \emph{centralizer} of $\cc$.
The universality of $\{i_{X,Y}\}_{Y \in \cc}$ translates to a universal factorization property for  $\partial$ as follows: for any $X,M\in\cc$ and any natural transformation $\{\xi_Y\co X \otimes Y \to Y \otimes M\}_{Y \in \cc}$, there exists a unique morphism $r\co Z(X) \to M$ in $\cc$ such that $\xi_Y=(\id_Y \otimes r)\partial_{X,Y}$ for all $Y \in \cc$:
$$
\psfrag{A}[Br][Br]{\scalebox{.8}{$X$}}
\psfrag{B}[Bl][Bl]{\scalebox{.8}{$M$}}
\psfrag{Y}[Bl][Bl]{\scalebox{.8}{$Y$}}
\psfrag{Z}[Br][Br]{\scalebox{.8}{$Y$}}
\psfrag{r}[Bc][Bc]{\scalebox{.9}{$r$}}
\psfrag{n}[Bc][Bc]{\scalebox{.9}{$\xi_Y$}}
\rsdraw{.45}{.9}{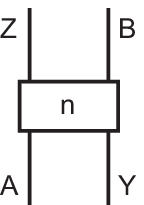} \;\,=\;\,
\rsdraw{.45}{.9}{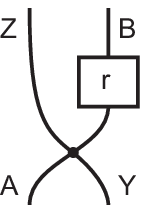} \,\; .
$$
Also, the parameter theorem for coends (see \cite{ML1}) implies that the family of coends $\{Z(X)\}_{X \in \cc}$ uniquely extend to a functor $Z \co \cc \to \cc$ so that~$\partial=\{\partial_{X,Y}\}_{X,Y \in \cc}$ is natural in $X$ and~$Y$.

By \cite[Corollary 5.14 and Theorem 6.5]{BV3}, the functor $Z$ has the structure of a quasitriangular Hopf monad, called the \emph{central Hopf monad} of $\cc$, which describes the center $\zz(\cc)$ of $\cc$ (meaning that the Eilenberg-Moore category of $Z$ is isomorphic to $\zz(\cc)$ as braided monoidal categories). The product $m$, unit $u$, and comonoidal structure $(Z_2,Z_0)$  are characterized (using the universal factorization property for~$\partial$) by the following equalities
with $X,X_1,X_2,Y,Y_1,Y_2\in\cc$:
\begin{center}
\psfrag{A}[Bc][Bc]{\scalebox{.8}{$Y_1$}}
\psfrag{B}[Bc][Bc]{\scalebox{.8}{$Y_2$}}
\psfrag{X}[Bc][Bc]{\scalebox{.8}{$X$}}
\psfrag{U}[Bc][Bc]{\scalebox{.8}{$Y_1 \otimes Y_2$}}
\psfrag{Z}[Bc][Bc]{\scalebox{.8}{$Z(X)$}}
\psfrag{r}[Bc][Bc]{\scalebox{.9}{$m_X$}}
\rsdraw{.45}{.9}{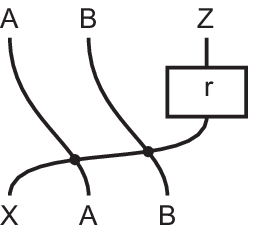} \, $=$ \, \rsdraw{.45}{.9}{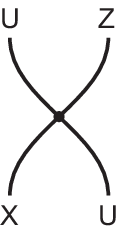},
 \qquad \qquad
\psfrag{X}[Bc][Bc]{\scalebox{.8}{$X$}}
\psfrag{U}[Bc][Bc]{\scalebox{.8}{$\un$}}
\psfrag{Z}[Bc][Bc]{\scalebox{.8}{$Z(X)$}} $u_X
= $ \rsdraw{.45}{.9}{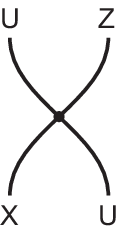},  \\[1em]
\psfrag{A}[Bc][Bc]{\scalebox{.8}{$Z(X_1)$}}
\psfrag{B}[Bc][Bc]{\scalebox{.8}{$Z(X_2)$}}
\psfrag{Y}[Bc][Bc]{\scalebox{.8}{$Y$}}
\psfrag{X}[Bc][Bc]{\scalebox{.8}{$X_1 \otimes X_2$}}
\psfrag{E}[Bc][Bc]{\scalebox{.8}{$X_1$}}
\psfrag{H}[Bc][Bc]{\scalebox{.8}{$X_2$}}
\psfrag{T}[Bc][Bc]{\scalebox{.9}{$Z_2(X_1,X_2)$}}
\rsdraw{.45}{.9}{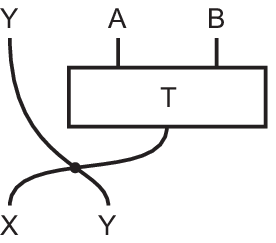} \, $=$ \, \rsdraw{.45}{.9}{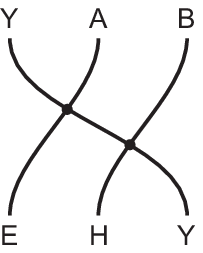},
\qquad \qquad \psfrag{Y}[Bc][Bc]{\scalebox{.8}{$Y$}}
\psfrag{C}[Bc][Bc]{\scalebox{.8}{$Z(\un)$}}
\psfrag{r}[Bc][Bc]{\scalebox{.9}{$Z_0$}}
\rsdraw{.45}{.9}{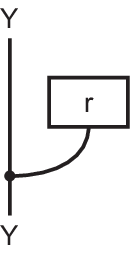} \, $=$ \, \rsdraw{.45}{.9}{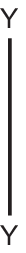}\;.
\end{center}
Note that the left and right antipodes and $R$-matrix of $Z$ can similarly be described (see~\cite{BV3}), but we do not recall these descriptions since we do not use them in the sequel.

\subsection{Proof of Theorem~\ref*{thm-chromatic}}\label{sect-proof-chromatic}
Note that a left chromatic map  based at a projective object $P$ for a projective generator $G$ is nothing but a right chromatic map based at $P$ for $G$ in the finite tensor category $\cc^{\otimes \opp}=(\cc,\otimes^\opp,\un)$. Thus we only need to prove the existence of right chromatic maps.

Since  the category $\cc$ has a projective generator, the coend \eqref{eq-coend-Z} exists for all $X \in \cc$ (by \cite[Lemma 5.1.8]{KL}). By Section~\ref{sect-central-Hopf-monad}, we can then consider the central Hopf monad $Z=(Z,m,u,Z_2,Z_0)$ of~$\cc$ and its associated centralizer $\partial=\{\partial_{X,Y}\co X \otimes Y \to Y \otimes Z(X)\}_{X,Y \in \cc}$.

Recall the natural transformation $\Lambda^r=\{\Lambda_Y^r \co \alpha \otimes Y \to Y\}_{Y \in \cc}$ from Lemma~\ref{lem-Lambda}. The universal factorization property for  $\partial$ gives that there is a unique morphism $\Lambda_r\co Z(\alpha) \to \un$ in $\cc$ such that $\Lambda_Y^r=(\id_Y \otimes \Lambda_r)\partial_{\alpha,Y}$ for all $Y\in \cc$:
$$
\psfrag{Y}[Bl][Bl]{\scalebox{.8}{$Y$}}
\psfrag{e}[Br][Br]{\scalebox{.8}{$\alpha$}}
\psfrag{Z}[Br][Br]{\scalebox{.8}{$Y$}}
\psfrag{r}[Bc][Bc]{\scalebox{.9}{$\Lambda_r$}}
\psfrag{n}[Bc][Bc]{\scalebox{.9}{$\Lambda^r_Y$}}
\rsdraw{.45}{.9}{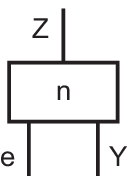} \;\,=\;\,
\rsdraw{.45}{.9}{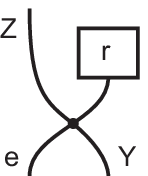} \,\; .
$$

\begin{lem}\label{lem-Ar-integ}
The morphism $\Lambda_r\co Z(\alpha) \to \un$ is a nonzero right $\alpha$-integral for $Z$.
\end{lem}

\begin{proof}
Clearly $\Lambda_r\neq 0$ since $\Lambda^r$ is nonzero (because $\Lambda^r_{P_0}=\eta \otimes \varepsilon\neq 0$).
We need to prove that $(\Lambda_r \otimes \id_{Z(\un)})Z_2(\alpha,\un)= u_\un  \Lambda_r$. It follows from the universal factorization property for~$\partial$ that this amounts to showing the equality of the natural transformations $l=\{l_Y\}_{Y \in \cc}$ and $r=\{r_Y\}_{Y \in \cc}$ defined by
$$
l_Y=\bigl(\id_Y \otimes( \Lambda_r \otimes \id_{Z(\un)})Z_2(\alpha,\un)\bigr)\partial_{\alpha,Y} \quad \text{and} \quad r_Y=(\id_Y \otimes u_\un  \Lambda_r)\partial_{\alpha,Y}.
$$
Note that the definitions of $\Lambda_r$ and $Z_2(\alpha,\un)$ imply that $ r_Y=(\id_Y \otimes u_\un)\Lambda^r_Y$ and
$$
l_Y=\bigl((\id_Y \otimes \Lambda_r)\partial_{\alpha,Y} \otimes \id_{Z(\un)}\bigr)(\id_\alpha \otimes \partial_{\un,Y})
= (\Lambda^r_Y \otimes \id_{Z(\un)})(\id_\alpha \otimes \partial_{\un,Y}).
$$
Then $l_P=0=r_P$ for any indecomposable projective object $P$ non isomorphic to~$P_0$ (since
$\Lambda^r_P=0$). Also
$$
l_{P_0}\overset{(i)}{=} \eta \otimes \bigl((\varepsilon \otimes \id_{Z(\un)})\partial_{\un,P_0}\bigr) \overset{(ii)}{=} \eta \otimes \partial_{\un,\un}\varepsilon\overset{(iii)}{=} (\id_{P_0} \otimes u_\un)(\eta \otimes \varepsilon) \overset{(iv)}{=}r_{P_0}.
$$
Here  $(i)$ and $(iv)$ follow from the equality $\Lambda^r_{P_0}=\eta \otimes \varepsilon$, $(ii)$ from the naturality of~$\partial$, and $(iii)$ from the definition of $u_\un$. Consequently, using that any projective object is a (finite) direct sum of indecomposable projective objects, we obtain that $l_P=r_P$ for all $P \in \Proj$. Finally we conclude that $l=r$ by applying Lemma~\ref{lem-extension} with the functors $F=\alpha \otimes -$  and $G=-\otimes Z(\un)$.
\end{proof}

Since the central Hopf monad $Z$ is the central Hopf comonad for the finite tensor category $\cc^\opp$ opposite to $\cc$, it follows from Lemma~\ref{lem-Ar-integ} and \cite[Theorem 4.8]{Sh2} that  there is a unique $\alpha$-cointegral $\lambda \co \un \to Z(\alpha)$ such that $\Lambda_r\lambda=\id_\un$. 

\begin{lem}\label{lem-lambda-Lambda-id}
For any $X\in\cc$, $(\id_X \otimes \Lambda_rm_\alpha)\partial_{Z(\alpha),X}(\lambda \otimes \id_X)=\id_X$.
\end{lem}
\begin{proof}
We have:
\begin{gather*}
(\id_X \otimes \Lambda_rm_\alpha)\partial_{Z(\alpha),X}(\lambda \otimes \id_X) \overset{(i)}{=}
(\id_X \otimes \Lambda_rm_\alpha Z(\lambda))\partial_{\un,X} \\ \overset{(ii)}{=} (\id_X \otimes \Lambda_r\lambda Z_0)\partial_{\un,X} \overset{(iii)}{=} (\id_X \otimes Z_0)\partial_{\un,X}  \overset{(iv)}{=}\id_X.
\end{gather*}
Here $(i)$ follows from the naturality of~$\partial$, $(ii)$ from the fact that $m_\alpha Z(\lambda)= \lambda Z_0$ (because $\lambda$ is an $\alpha$-cointegral),  $(iii)$ from the equality $\lambda\Lambda_r=\id_\un$, and $(iv)$ from the definition of $Z_0$.
\end{proof}

Let $P$ be a projective object and $G$ be a  projective generator of~$\cc$. Set
$$
a_P=(\id_P \otimes \rev_{G} \otimes \id_{Z(\alpha)})(\id_{P\otimes G} \otimes \partial_{\alpha,\rdual{G}}) \co P \otimes G \otimes \alpha \otimes \rdual{G} \to P \otimes Z(\alpha).
$$
Graphically,
$$
\psfrag{P}[Br][Br]{\scalebox{.8}{$P$}}
\psfrag{G}[Br][Br]{\scalebox{.8}{$G$}}
\psfrag{a}[Br][Br]{\scalebox{.8}{$\alpha$}}
\psfrag{H}[Bl][Bl]{\scalebox{.8}{$\rdual{G}$}}
\psfrag{Z}[Bl][Bl]{\scalebox{.8}{$Z(\alpha)$}}
\psfrag{u}[Bc][Bc]{\scalebox{.9}{$\rev_{G}$}}
a_P=\rsdraw{.45}{.9}{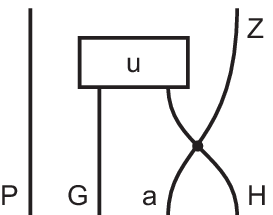}\;.
$$
\begin{lem}\label{lem-ap-epi}
$a_P$ is an epimorphism.
\end{lem}
\begin{proof}
Since $\rdual{G}$ is a projective generator of $\cc$, the universal dinatural transformation $i_{\alpha,\rdual{G}}\co \ldual{}(\rdual{G}) \otimes \alpha \otimes \rdual{G} \to Z(\alpha)$ is an epimorphism (by \cite[Corollary 5.1.8]{KL}). Then $b_P=\id_P \otimes i_{\alpha,\rdual{G}}$ is an epimorphism (since $\otimes$ is exact because $\cc$ is rigid).
Considering the isomorphism $\varphi_G=(\rev_G \otimes \id_{\ldual{}(\rdual{G})})(\id_G \otimes \lcoev_{\rdual{G}})\co G \to \ldual{}(\rdual{G})$, we conclude that $a_P=b_P(\id_P \otimes \varphi_G \otimes \id_{\alpha \otimes \rdual{G}})$ is an epimorphism.
\end{proof}

Since $a_P$ is an epimorphism (by Lemma~\ref{lem-ap-epi}) and $P$ is a projective object, the morphism $\id_P \otimes \lambda \co P \to P \otimes Z(\alpha)$ factors through $a_P$, that is, $\id_P \otimes \lambda=a_Pd_P$ for some morphism $d_P \co P \to P \otimes G \otimes \alpha \otimes \rdual{G}$. Set
$$
\chr_P^r=(\id_{} \otimes \rev_{\rdual{G}})(d_P \otimes \id_{\rrdual{G}}) \co P \otimes \rrdual{G} \to P \otimes G \otimes \alpha.
$$
Graphically,
$$
\psfrag{P}[Br][Br]{\scalebox{.8}{$P$}}
\psfrag{G}[Br][Br]{\scalebox{.8}{$G$}}
\psfrag{a}[Br][Br]{\scalebox{.8}{$\alpha$}}
\psfrag{H}[Br][Br]{\scalebox{.8}{$\rdual{G}$}}
\psfrag{Z}[Bl][Bl]{\scalebox{.8}{$Z(\alpha)$}}
\psfrag{X}[Bl][Bl]{\scalebox{.8}{$\rrdual{G}$}}
\psfrag{u}[Bc][Bc]{\scalebox{.9}{$\rev_{\rdual{G}}$}}
\psfrag{v}[Bc][Bc]{\scalebox{.9}{$d_P$}}
\chr_P^r=\rsdraw{.45}{.9}{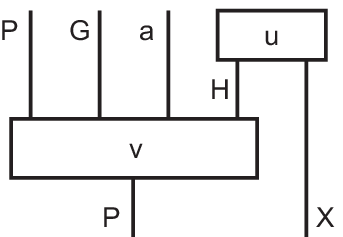}\;.
$$
Then $\chr_P^r$ is a right chromatic map based at $P$ for $G$. Indeed, for any $X \in \cc$,
\begin{gather*}
\psfrag{P}[Br][Br]{\scalebox{.8}{$P$}}
\psfrag{G}[Br][Br]{\scalebox{.8}{$G$}}
\psfrag{a}[Br][Br]{\scalebox{.8}{$\alpha$}}
\psfrag{H}[Bl][Bl]{\scalebox{.8}{$\rrdual{G}$}}
\psfrag{X}[Bl][Bl]{\scalebox{.8}{$X$}}
\psfrag{Q}[Bl][Bl]{\scalebox{.8}{$P$}}
\psfrag{R}[Br][Br]{\scalebox{.8}{$\rdual{G}$}}
\psfrag{u}[Bc][Bc]{\scalebox{.9}{$\rev_{G}$}}
\psfrag{e}[Bc][Bc]{\scalebox{.9}{$\rcoev_{\rdual{G}}$}}
\psfrag{c}[Bc][Bc]{\scalebox{.9}{$\chr_P^r$}}
\psfrag{A}[Bc][Bc]{\scalebox{.9}{$\Lambda^r_{\rdual{G} \otimes X}$}}
\rsdraw{.45}{.9}{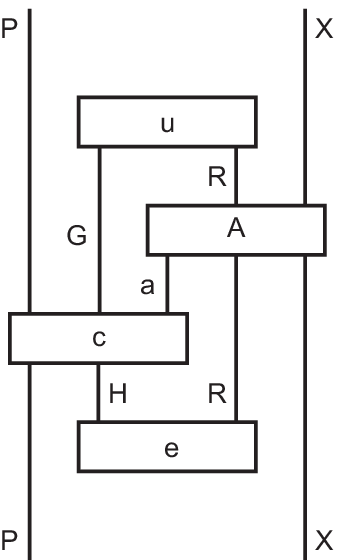}\;
\psfrag{u}[Bc][Bc]{\scalebox{.9}{$\rev_{G}$}}
\psfrag{e}[Bc][Bc]{\scalebox{.9}{$d_P$}}
\psfrag{c}[Bc][Bc]{\scalebox{.9}{$\id_{\rdual{G} \otimes X}$}}
\psfrag{r}[Bc][Bc]{\scalebox{.9}{$\Lambda_r$}}
\psfrag{R}[Bl][Bl]{\scalebox{.8}{$\rdual{G}$}}
\psfrag{X}[Bl][Bl]{\scalebox{.8}{$X$}}
\psfrag{P}[Br][Br]{\scalebox{.8}{$P$}}
\psfrag{G}[Br][Br]{\scalebox{.8}{$G$}}
\psfrag{a}[Br][Br]{\scalebox{.8}{$\alpha$}}
\overset{(i)}{=}\;\rsdraw{.45}{.9}{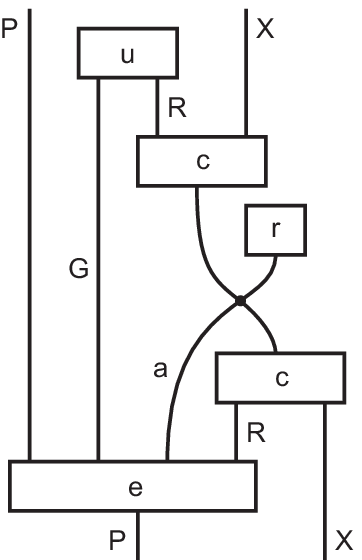}
\psfrag{u}[Bc][Bc]{\scalebox{.9}{$\rev_{G}$}}
\psfrag{e}[Bc][Bc]{\scalebox{.9}{$d_P$}}
\psfrag{r}[Bc][Bc]{\scalebox{.9}{$\Lambda_r m_\alpha$}}
\psfrag{R}[Bl][Bl]{\scalebox{.8}{$\rdual{G}$}}
\psfrag{X}[Bl][Bl]{\scalebox{.8}{$X$}}
\psfrag{P}[Br][Br]{\scalebox{.8}{$P$}}
\psfrag{G}[Br][Br]{\scalebox{.8}{$G$}}
\psfrag{a}[Br][Br]{\scalebox{.8}{$\alpha$}}
\psfrag{Y}[Bl][Bl]{\scalebox{.8}{$\rdual{G} \otimes X$}}
\overset{(ii)}{=}\;\rsdraw{.45}{.9}{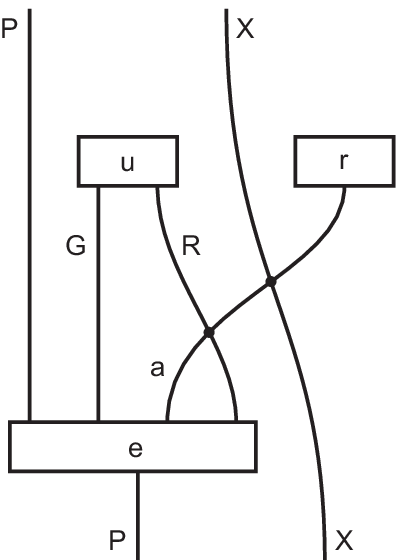}\\[1em]
\psfrag{c}[Bc][Bc]{\scalebox{.9}{$a_P$}}
\psfrag{e}[Bc][Bc]{\scalebox{.9}{$d_P$}}
\psfrag{r}[Bc][Bc]{\scalebox{.9}{$\Lambda_r m_\alpha$}}
\psfrag{R}[Br][Br]{\scalebox{.8}{$\rdual{G}$}}
\psfrag{X}[Bl][Bl]{\scalebox{.8}{$X$}}
\psfrag{Y}[Br][Br]{\scalebox{.8}{$X$}}
\psfrag{P}[Br][Br]{\scalebox{.8}{$P$}}
\psfrag{G}[Br][Br]{\scalebox{.8}{$G$}}
\psfrag{a}[Br][Br]{\scalebox{.8}{$\alpha$}}
\overset{(iii)}{=}\;\rsdraw{.45}{.9}{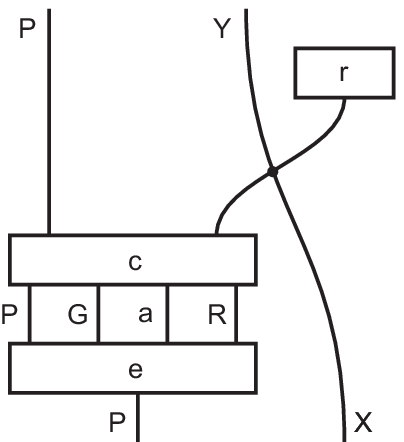}\;
\psfrag{c}[Bc][Bc]{\scalebox{.9}{$\lambda$}}
\psfrag{r}[Bc][Bc]{\scalebox{.9}{$\Lambda_r m_\alpha$}}
\psfrag{X}[Bl][Bl]{\scalebox{.8}{$X$}}
\psfrag{Y}[Br][Br]{\scalebox{.8}{$X$}}
\psfrag{P}[Br][Br]{\scalebox{.8}{$P$}}
\overset{(iv)}{=}\;\rsdraw{.45}{.9}{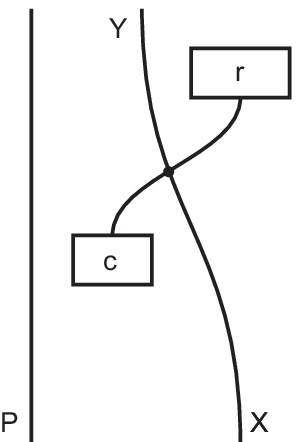}\;
\psfrag{X}[Bl][Bl]{\scalebox{.8}{$X$}}
\psfrag{P}[Br][Br]{\scalebox{.8}{$P$}}
\overset{(v)}{=}\;\rsdraw{.45}{.9}{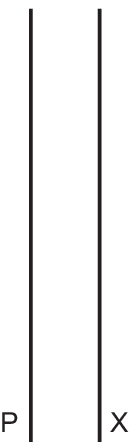}\;.
\end{gather*}
Here  $(i)$ follows from the definitions of $\chr_P^r$ and $\Lambda_r$, $(ii)$ from the definition of the product $m$ of $Z$, $(iii)$ from the definition of $a_P$, $(iv)$ from the fact that $a_Pd_P=\id_P \otimes \lambda$, and $(v)$ from Lemma~\ref{lem-lambda-Lambda-id}.

\subsection{Proof of Lemma~\ref*{lem-Lambdas-comparison}}\label{proof-lem-Lambdas-comparison}
Since $\Hom_\cc(P_0,\un)=\kk \, \varepsilon$,  $\Hom_\cc(\un,P_0)=\kk\,\eta$,  and $\mt_{P_0}(\eta \varepsilon)=1_\kk$, the definition of $\Lambda^\mt$ gives that $\Lambda^\mt_{P_0}=\eta \varepsilon$. Using that $\eta \varepsilon=\eta \otimes \varepsilon= \varepsilon \otimes \eta$, we get that $\Lambda^\mt_{P_0}=\Lambda^r_{P_0}=\Lambda^l_{P_0}$.

Let $P$ be an indecomposable projective object non isomorphic to $P_0$. Since $\Hom_\cc(P,\un)=0=\Hom_\cc(\un,P)$, the definition of $\Lambda^\mt$ gives that $\Lambda^\mt_P=0$, and so we get that $\Lambda^\mt_P=\Lambda^r_P=\Lambda^l_P$.

Since any projective object is a (finite) direct sum of indecomposable projective objects, the above equalities together with the naturality $\Lambda^\mt$, $\Lambda^r$, $\Lambda^l$ implies that $\Lambda^\mt_P=\Lambda^r_P=\Lambda^l_P$ for all $P \in \Proj$.

\subsection{Chromatic maps in $H$-mod}\label{proof-chromatic-H-mod}
We first prove that $\Lambda$ defined in Example \ref{exa-H-mod}
is a non zero left cointegral of $H$.
It follows from  \cite[Proposition 10.6.2.]{R2012} that the set  $L_\alpha=\{x\in H \, | \, hx=\alpha_H(h)x\}$ is a one
dimensional left ideal of $H$ which is equal to the set of right cointegrals of~$H$.
Since $S^{-1}(\Lambda)=i\eta(1_\kk)\neq0$ and $i\eta$ is $H$-linear,
$L_\alpha$ is generated by $S^{-1}(\Lambda)$ which is then a nonzero right
cointegral.  Consequently, $\Lambda=SS^{-1}(\Lambda)$ is a nonzero left cointegral.

Let us now prove that the  $\kk$-linear  homomorphism
$$
 \chr_H^l \co \left \{\begin{array}{ccl} \lldual{H}\otimes H & \to & \alpha\otimes H\otimes H \\ e_x \otimes y & \mapsto &  \lambda(S(y_{(1)})x)\alpha_{H}(y_{(2)})\otimes y_{(3)}\otimes y_{(4)}
  \end{array}\right. 
$$
from Example~\ref{exa-H-mod} is a  left chromatic map based at $H$ for $H$. Notice first that $\chr_H^l$ is $H$-linear. Indeed, for any $x,y,h \in H$,
\begin{align*}
\chr_H^l \bigl(h \cdot(e_x \otimes y)\bigr) 
&\overset{(i)}{=} \chr_H^l (e_{S^2(h_{(1)})x} \otimes h_{(2)}y) \\
&\overset{(ii)}{=}  \lambda\bigl(S(h_{(2)}y_{(1)}) S^2(h_{(1)})x\bigr)\, \alpha_{H}( h_{(3)}y_{(2)})\otimes  h_{(4)}y_{(3)}\otimes  h_{(5)}y_{(4)}   \\
&\overset{(iii)}{=} \lambda\bigl(S\bigl(S(h_{(1)})h_{(2)}y_{(1)}\bigr) x\bigr)\, \alpha_{H}( h_{(3)}y_{(2)})\otimes  h_{(4)}y_{(3)}\otimes  h_{(5)}y_{(4)}   \\
&\overset{(iv)}{=} \lambda\bigl(S(y_{(1)}) x\bigr)\, \alpha_{H}( h_{(1)}y_{(2)})\otimes  h_{ (2)}y_{(3)}\otimes  h_{(3)}y_{(4)}   \\
&\overset{(v)}{=} \alpha_{H}( h_{(1)})\, \lambda\bigl(S(y_{(1)}) x\bigr) \, \alpha_{H}( y_{(2)})\otimes  h_{(2)}y_{(3)}\otimes  h_{(3)}y_{(4)} \\   
& \overset{(vi)}{=}  h \cdot \chr_H^l (e_x \otimes y).
\end{align*}
Here $(i)$ follows from the definition of the monoidal product in $H$-mod, $(ii)$ from the definition of $\chr_H^l$ and the multiplicativity of the coproduct of $H$,  $(iii)$ from the anti-multiplicativity of $S$, $(iv)$ from the axiom of the antipode, $(v)$ from
the multiplicativity of $\alpha_H$, and $(vi)$ from the definitions of $\chr_H^l$ and of the monoidal product in $H$-mod.

We next compute the natural transformation $\Lambda^l$. 
For any finite dimensional $H$-module~$M$, consider the $\kk$-linear  homomorphism
$$
 \tilde{\Lambda}_M^l \co \left \{\begin{array}{ccl} M\otimes\alpha & \to & M \\ m \otimes 1_\kk & \mapsto &  S^{-1}(\Lambda) \cdot m.
  \end{array}\right. 
$$
Then $\tilde{\Lambda}_M^l$ is $H$-linear. Indeed, for any $h \in H$ and $m \in M$,
\begin{gather*}
\tilde{\Lambda}^l_M(h \cdot(m\otimes 1_\kk)) \overset{(i)}{=}
\alpha_H(h_{(2)}) \, S^{-1}(\Lambda)h_{(1)} \cdot m \overset{(ii)}{=} \varepsilon_H(h_{(1)})\, \alpha_H(h_{(2)}) \, S^{-1}(\Lambda) \cdot m \\
\overset{(iii)}{=} \alpha_H(h)\, S^{-1}(\Lambda) \cdot m  
\overset{(iv)}{=}S^{-1}(\Lambda S(h)) \cdot m
\overset{(v)}{=} hS^{-1}(\Lambda) \cdot m 
 \overset{(vi)}{=}  h \cdot \tilde{\Lambda}^l_M(m\otimes 1_\kk),
\end{gather*}
where $\varepsilon_H$ is the counit of $H$.  Here $(i)$ follows from
the definitions of $ \tilde{\Lambda}_M^l$ and of the action of
$M \otimes \alpha$, $(ii)$ from the fact that $S^{-1}(\Lambda)$
is a right cointegral of $H$, $(iii)$ from the counitality of the
coproduct, $(iv)$ from the property characterizing $\alpha_H$ (see
Example~\ref{exa-H-mod}), $(v)$ from the anti-multiplicativity of $S$,
and $(vi)$ from the definition of~$\tilde{\Lambda}_M^l$. Clearly, the
family $\{\tilde{\Lambda}_M^l\}_M$ is natural in $M$. Now for
any $h\in H$,
$$
\tilde{\Lambda}_H^l(h\otimes 1_\kk)\overset{(i)}{=}S^{-1}(\Lambda)h\overset{(ii)}{=}\varepsilon_H(h)i\eta(1_\kk)\overset{(iii)}{=}(\varepsilon p\otimes
i\eta)(h\otimes 1_\kk).
$$
Here $(i)$ follows from the definition of $\tilde{\Lambda}_H^l$, $(ii)$ from the fact that $S^{-1}(\Lambda)$ is a right cointegral and from the definition of $\Lambda$, and $(iii)$ from the definition of $p$.
Thus, using that $pi=\id_{P_0}$ and the naturality of $\tilde{\Lambda}^l$, we obtain:
$$
\tilde{\Lambda}_{P_0}^l=\tilde{\Lambda}_{P_0}^l(pi \otimes \id_\alpha)=p\tilde{\Lambda}_{H}^l(i\otimes\id_\alpha)=\varepsilon
pi\otimes pi\eta=\varepsilon\otimes \eta.
$$
Also, for any indecomposable projective object $P$ non isomorphic
to $P_0$, we have $\tilde{\Lambda}_{P}^l=0$ since in the projective generator $H$, 
the image $\kk S^{-1}(\Lambda)=i\eta(\kk)\subset i(P_0)$ of $\tilde{\Lambda}_{H}^l$ is isomorphic to the simple $H$-module $\alpha$, and $i(P_0) \cong P_0$ is the only (up to isomorphism) indecomposable projective $H$-module which has a submodule isomorphic to $\alpha$ (by uniqueness of the socle, see Section~\ref{sect-finite-tensor}).
Consequently, the uniqueness in Lemma~\ref{lem-Lambda} implies that $\Lambda^l=\tilde{\Lambda}^l$.

Pick a projective cover $\varepsilon\co P_0 \to \kk$  of the unit object and a  monomorphism $\eta\co \alpha \to P_0$.
Since the counit
$\varepsilon_H\co H\to\kk$ of $H$ is an epimorphism, there exists an epimorphism
$p\co H\to P_0$ such that $\varepsilon_H=\varepsilon p$. Let $i\co P_0\to H$
be a section of $p$ in $H$-mod and set $\Lambda=S(i\eta(1_\kk))\in H$.
We prove in Section~\ref{proof-chromatic-H-mod} that $\Lambda$ is
a non zero left cointegral. By \cite[Theorem 10.2.2]{R2012}, there is a unique right integral
$\lambda\in H^*$ such that $\lambda(\Lambda)=1$.

We now prove that $\chr_H^l$ is a left chromatic map. Let $M$ be a finite dimensional $H$-module. Pick any $m \in M$ and $x \in H$. In $M\otimes\ldual H\otimes\alpha\otimes H\otimes H$, we have:
$$
(\id_{M\otimes\ldual{H}} \otimes \chr_H^l) (\id_{M}\otimes \lcoev_{\ldual{H}} \otimes
\id_H)(m\otimes x)=  m\otimes\lambda(S(x_{(1)})\_)\otimes\alpha_{H}(x_{(2)})\otimes
x_{(3)}\otimes x_{(4)}.
$$
Evaluating this vector under $(\id_M \otimes \lev_H \otimes \id_H)(\Lambda^l_{M \otimes \ldual{H}} \otimes \id_{H \otimes H})$ gives
\begin{align*}
\alpha_{H}&(x_{(2)})\, \lambda\bigl(S(x_{(1)})\, \Lambda_{(1)}x_{(3)}\bigr)\, \bigl( S^{-1}(\Lambda_{(2)})\cdot m \bigr) \otimes x_{(4)} \\
& \overset{(i)}{=} \alpha_{H}(x_{(2)})\, \alpha_{H}\bigl(S(x_{(3)})\bigr)\,\lambda\bigl(S^2(x_{(4)})S(x_{(1)})\, \Lambda_{(1)}\bigr)\, \bigl( S^{-1}(\Lambda_{(2)})\cdot m \bigr) \otimes x_{(5)} \\
& \overset{(ii)}{=} \lambda\bigl(S^2(x_{(2)})S(x_{(1)})\, \Lambda_{(1)}\bigr)\, \bigl( S^{-1}(\Lambda_{(2)})\cdot m \bigr) \otimes x_{(3)}\\
&  \overset{(iii)}{=} \varepsilon_H(x_{(1)}) \, \lambda(\Lambda_{(1)})\, \bigl( S^{-1}(\Lambda_{(2)})\cdot m \bigr) \otimes x_{(2)} \\
& \overset{(iv)}{=}\, \bigl( S^{-1}\bigl(\lambda(\Lambda_{(1)})\Lambda_{(2)}\bigr)\cdot m \bigr) \otimes x 
  \overset{(v)}{=}\, m \otimes x.
\end{align*}
Here $(i)$ follows from the fact that $\lambda(ab)=\alpha_{H}(S(b_{(1)}))\lambda(S^{2}(b_{(2)})a)$ for all $a,b \in H$ (see~\cite[Theorem 10.5.4]{R2012}), $(ii)$ from multiplicativity of $\alpha_H$ and the axiom of the antipode, $(iii)$ from the axiom of the antipode, $(iv)$ from the counitailty of the coproduct, and $(v)$ from the fact that $\lambda(\Lambda_{(1)})\Lambda_{(2)}=\lambda(\Lambda) 1_H=1_H$.
Consequently, 
$$
(\id_M \otimes \lev_H \otimes \id_H)(\Lambda^l_{M \otimes \ldual{H}} \otimes \id_{H \otimes H})(\id_{M \otimes \ldual{H}} \otimes \chr_H^l) (\id_M \otimes \lcoev_{\ldual{H}} \otimes \id_H)=\id_{M \otimes H},
$$
that is, $\chr_H^l$ is a left chromatic map based at $H$ for $H$.

Finally, the expression for $\chr_H^r$ is derived from that of $\chr_H^l$ by noticing that for any projective generator $G$ and projective object $P$ in $H$-mod, a right chromatic map based at $P$ for $G$ in $H$-mod is a left chromatic map at $P$ for $G$ in $(H$-mod$)^{\otimes \opp}$, that is, in $(H^{\mathrm{cop}})$-mod, where $H^{\mathrm{cop}}$ is $H$ with opposite coproduct (for which $S^{\mathrm{cop}}=S^{-1}$, $\Lambda^{\mathrm{cop}}=\Lambda$, $\lambda^{\mathrm{cop}}=\lambda S$, and $\alpha_{H^{\mathrm{cop}}}=\alpha_H$).

\subsubsection*{Acknowledgments}  F.C.\ is supported by Labex CIMI (ANR 11-LABX-0040) and the ANR Project CATORE (ANR-18-CE40-0024).
N.G.\ is supported by the Labex CEMPI (ANR-11-LABX-0007-01), IMT Toulouse,  and the NSF grant DMS-2104497.    A.V.\ is supported by the FNS-ANR grant OCHoTop (ANR-18-CE93-0002) and the Labex CEMPI (ANR-11-LABX-0007-01).


\begin{thebibliography}{EGNO}
\bibitem[BBG]{BBG} Beliakova, A., Blanchet, C., Gainutdinov,
  A. M. \emph{Modified trace is a symmetrised integral} Selecta
  Math. (N.S.), 2021, 27, Paper No. 31, 51.

\bibitem[BLV]{BLV} Brugui{\`e}res, A., Lack, S., Virelizier, A.,
  \emph{Hopf monads on monoidal categories}, Adv. in Math. 227 (2011),
  745--800.

\bibitem[BV1]{BV2}
Brugui{\`e}res, A.,  Virelizier, A.,  \emph{Hopf monads}, Adv. in Math.  215  (2007), 679--733.

\bibitem[BV2]{BV3}
\bysame,  \emph{Quantum double of {H}opf monads and categorical centers},
 Trans. Amer. Math. Soc. 364 (2012),   1225--1279.

\bibitem[CGHP]{CGHP} Costantino, F., Geer, N., Ha\"ioun B.,
  Patureau-Mirand, B., \emph{Skein (3+1)-TQFTs from non semisimple
    ribbon categories}, preprint.

\bibitem[CGPT]{CGPT20}
Costantino, F., Geer, N., Patureau-Mirand, B., Turaev, V., \emph{Kuperberg and Turaev-Viro Invariants in Unimodular
  Categories}, Pacific J. Math. 306 (2020), 421--450.

\bibitem[CGPV]{CGPV23}
Costantino, F., Geer, N., Patureau-Mirand, B., Virelizier, A., \emph{Non compact (2+1)-TQFTs from non-semisimple spherical categories}, preprint (2023) arXiv:2302.04509.

\bibitem[EGNO]{EGNO}
Etingof, P.,   Gelaki, S., Nikshych, D.,  Ostrik, V.,  Tensor categories,  Mathematical Surveys and Monographs, 205. American Mathematical Society, Providence, RI, 2015.

\bibitem[GKP]{GKP22}
Geer, N., Kujawa, J., Patureau-Mirand, B., \emph{M-traces in (non-unimodular) pivotal categories}
  Algebr. Represent. Theory 25 (2022), 759--776.

\bibitem[GPV]{GPV13}
Geer, N., Patureau-Mirand, B., Virelizier, A., \emph{Traces on ideals in pivotal categories}.
Quantum Topol. 9 (2013), 91--124.

\bibitem[KL]{KL}
Kerler, T.,  Lyubashenko, V.,  Non-semisimple topological quantum field theories for 3-manifolds with corners. Lecture Notes in Math., 1765. Springer-Verlag, Berlin, 2001.

\bibitem[ML]{ML1}
MacLane, S., \emph{Categories for the working mathematician},
Second edition,   Springer-Verlag, New York, 1998.

\bibitem[Ra]{R2012} Radford, D. E., \emph{Hopf algebras}, World
  Scientific Publishing Co. Pte. Ltd., Hackensack, NJ, (2012), 49,
  xxii+559 pp.

\bibitem[Sh]{Sh2}
Shimizu, K., \emph{Integrals for Finite Tensor Categories}
Algebr. Represent. Theor. 22 (2019), 459--493.

\bibitem[SS]{SS2021} Shibata, T., Shimizu, K.,   \emph{Modified traces and
    the Nakayama functor}, Algebr. Represent. Theory 26 (2023), 513--551.

\bibitem[Tu]{Tu0} Turaev, V., \emph{Quantum invariants of knots and
    3-manifolds}, de Gruyter Studies in Mathematics, 18. Walter de
  Gruyter, Berlin, 1994.

\bibitem[TVi]{TVi5}
Turaev, V., Virelizier, A., \emph{Monoidal Categories and Topological Field Theory}, Progress in  Mathematics, 322. Birkh\"auser, Basel, 2017. xii+523 pp.
\end{thebibliography}
\end{document}